\documentclass[11pt, a4paper, twoside,leqno]{amsart}
\usepackage[centering, totalwidth = 390pt, totalheight = 600pt]{geometry}
\usepackage{amssymb, amsmath, amsthm}
\usepackage{enumerate, microtype, stmaryrd, url} 
\usepackage[latin1]{inputenc}
\usepackage{color}
\definecolor{darkgreen}{rgb}{0,0.45,0}
\usepackage[colorlinks,citecolor=darkgreen,linkcolor=darkgreen]{hyperref}
\usepackage[british]{babel}

\DeclareMathAlphabet{\mathbf}{OT1}{cmr}{b}{n}

\usepackage[arrow, matrix, tips, curve, graph, rotate]{xy}
\SelectTips{cm}{10}

\makeatletter
\def\matrixobject@{%
  \edef \next@{={\DirectionfromtheDirection@ }}%
  \expandafter \toks@ \next@ \plainxy@
  \let\xy@@ix@=\xyq@@toksix@
  \xyFN@ \OBJECT@}
\let\xy@entry@@norm=\entry@@norm
\def\entry@@norm@patched{%
  \let\object@=\matrixobject@
  \xy@entry@@norm }
\AtBeginDocument{\let\entry@@norm\entry@@norm@patched}
\makeatother

\newcommand{\twocong}[2][0.5]{\ar@{}[#2] \save ?(#1)*{\cong}\restore}
\newcommand{\twoeq}[2][0.5]{\ar@{}[#2] \save ?(#1)*{=}\restore}
\newcommand{\rtwocell}[3][0.5]{\ar@{}[#2] \ar@{=>}?(#1)+/l 0.2cm/;?(#1)+/r 0.2cm/^{#3}}
\newcommand{\ltwocell}[3][0.5]{\ar@{}[#2] \ar@{=>}?(#1)+/r 0.2cm/;?(#1)+/l 0.2cm/^{#3}}
\newcommand{\ltwocello}[3][0.5]{\ar@{}[#2] \ar@{=>}?(#1)+/r 0.2cm/;?(#1)+/l 0.2cm/_{#3}}
\newcommand{\dtwocell}[3][0.5]{\ar@{}[#2] \ar@{=>}?(#1)+/u  0.2cm/;?(#1)+/d 0.2cm/^{#3}}
\newcommand{\dltwocell}[3][0.5]{\ar@{}[#2] \ar@{=>}?(#1)+/ur  0.2cm/;?(#1)+/dl 0.2cm/^{#3}}
\newcommand{\drtwocell}[3][0.5]{\ar@{}[#2] \ar@{=>}?(#1)+/ul  0.2cm/;?(#1)+/dr 0.2cm/^{#3}}
\newcommand{\dthreecell}[3][0.5]{\ar@{}[#2] \ar@3{->}?(#1)+/u  0.2cm/;?(#1)+/d 0.2cm/^{#3}}
\newcommand{\utwocell}[3][0.5]{\ar@{}[#2] \ar@{=>}?(#1)+/d 0.2cm/;?(#1)+/u 0.2cm/_{#3}}
\newcommand{\dtwocelltarg}[3][0.5]{\ar@{}#2 \ar@{=>}?(#1)+/u  0.2cm/;?(#1)+/d 0.2cm/^{#3}}
\newcommand{\utwocelltarg}[3][0.5]{\ar@{}#2 \ar@{=>}?(#1)+/d  0.2cm/;?(#1)+/u 0.2cm/_{#3}}

\newcommand{\pullbackcorner}[1][dr]{\save*!/#1-1.2pc/#1:(-1,1)@^{|-}\restore}

\newdir{(}{{}*!<0em,-.14em>-\cir<.14em>{l^r}}
\newdir{ (}{{}*!/-5pt/\dir{(}}
\newdir{ >}{{}*!/-5pt/\dir{>}}
\newcommand{\sh}[2]{**{!/#1 -#2/}}

\DeclareMathOperator{\ob}{ob}
\DeclareMathOperator{\colim}{colim}

\newcommand{\cat}[1]{\mathbf{#1}}

\newcommand{\id}{\mathrm{id}}
\newcommand{\thg}{{\mathord{\text{--}}}}

\newcommand{\defeq}{\mathrel{\mathop:}=}
\newcommand{\cd}[2][]{\vcenter{\hbox{\xymatrix#1{#2}}}}

\renewcommand{\phi}{\varphi}
\newcommand{\A}{{\mathcal A}}

\newcommand{\C}{{\mathcal C}}
\newcommand{\D}{{\mathcal D}}
\newcommand{\E}{{\mathcal E}}

\newcommand{\K}{{\mathcal K}}

\newcommand{\M}{{\mathcal M}}

\renewcommand{\P}{{\mathcal P}}

\let\sec\S
\renewcommand{\S}{{\mathcal S}}

\newcommand{\V}{{\mathcal V}}
\newcommand{\W}{{\mathcal W}}
\newcommand{\X}{{\mathcal X}}

\newcommand{\Z}{{\mathcal Z}}

\newcommand{\xtor}[1]{\cdl[@1]{{} \ar[r]|-{\object@{|}}^{#1} & {}}}
\newcommand{\tor}{\ensuremath{\relbar\joinrel\mapstochar\joinrel\rightarrow}}

\makeatletter

\def\hookleftarrowfill@{\arrowfill@\leftarrow\relbar{\relbar\joinrel\rhook}}
\def\twoheadleftarrowfill@{\arrowfill@\twoheadleftarrow\relbar\relbar}
\def\leftbararrowfill@{\arrowdoublefill@{\leftarrow\mkern-5mu}\relbar\mapstochar\relbar\relbar}
\def\Leftbararrowfill@{\arrowdoublefill@{\Leftarrow\mkern-2mu}\Relbar\Mapstochar\Relbar\Relbar}
\def\leftringarrowfill@{\arrowdoublefill@{\leftarrow\mkern-3mu}\relbar{\mkern-3mu\circ\mkern-2mu}\relbar\relbar}
\def\lefttriarrowfill@{\arrowfill@{\mathrel\triangleleft\mkern0.5mu\joinrel\relbar}\relbar\relbar}
\def\Lefttriarrowfill@{\arrowfill@{\mathrel\triangleleft\mkern1mu\joinrel\Relbar}\Relbar\Relbar}

\def\hookrightarrowfill@{\arrowfill@{\lhook\joinrel\relbar}\relbar\rightarrow}
\def\twoheadrightarrowfill@{\arrowfill@\relbar\relbar\twoheadrightarrow}
\def\rightbararrowfill@{\arrowdoublefill@{\relbar\mkern-0.5mu}\relbar\mapstochar\relbar\rightarrow}
\def\Rightbararrowfill@{\arrowdoublefill@{\Relbar\mkern-2mu}\Relbar\Mapstochar\Relbar\Rightarrow}
\def\rightringarrowfill@{\arrowdoublefill@\relbar\relbar{\mkern-2mu\circ\mkern-3mu}\relbar{\mkern-3mu\rightarrow}}
\def\righttriarrowfill@{\arrowfill@\relbar\relbar{\relbar\joinrel\mkern0.5mu\mathrel\triangleright}}
\def\Righttriarrowfill@{\arrowfill@\Relbar\Relbar{\Relbar\joinrel\mkern1mu\mathrel\triangleright}}

\def\leftrightarrowfill@{\arrowfill@\leftarrow\relbar\rightarrow}
\def\mapstofill@{\arrowfill@{\mapstochar\relbar}\relbar\rightarrow}

\renewcommand*\xleftarrow[2][]{\ext@arrow 20{20}0\leftarrowfill@{#1}{#2}}
\providecommand*\xLeftarrow[2][]{\ext@arrow 60{22}0{\Leftarrowfill@}{#1}{#2}}
\providecommand*\xhookleftarrow[2][]{\ext@arrow 10{20}0\hookleftarrowfill@{#1}{#2}}
\providecommand*\xtwoheadleftarrow[2][]{\ext@arrow 60{20}0\twoheadleftarrowfill@{#1}{#2}}
\providecommand*\xleftbararrow[2][]{\ext@arrow 10{22}0\leftbararrowfill@{#1}{#2}}
\providecommand*\xLeftbararrow[2][]{\ext@arrow 50{24}0\Leftbararrowfill@{#1}{#2}}
\providecommand*\xleftringarrow[2][]{\ext@arrow 10{26}0\leftringarrowfill@{#1}{#2}}
\providecommand*\xlefttriarrow[2][]{\ext@arrow 80{24}0\lefttriarrowfill@{#1}{#2}}
\providecommand*\xLefttriarrow[2][]{\ext@arrow 80{24}0\Lefttriarrowfill@{#1}{#2}}

\renewcommand*\xrightarrow[2][]{\ext@arrow 01{20}0\rightarrowfill@{#1}{#2}}
\providecommand*\xRightarrow[2][]{\ext@arrow 04{22}0{\Rightarrowfill@}{#1}{#2}}
\providecommand*\xhookrightarrow[2][]{\ext@arrow 00{20}0\hookrightarrowfill@{#1}{#2}}
\providecommand*\xtwoheadrightarrow[2][]{\ext@arrow 03{20}0\twoheadrightarrowfill@{#1}{#2}}
\providecommand*\xrightbararrow[2][]{\ext@arrow 01{22}0\rightbararrowfill@{#1}{#2}}
\providecommand*\xRightbararrow[2][]{\ext@arrow 04{24}0\Rightbararrowfill@{#1}{#2}}
\providecommand*\xrightringarrow[2][]{\ext@arrow 01{26}0\rightringarrowfill@{#1}{#2}}
\providecommand*\xrighttriarrow[2][]{\ext@arrow 07{24}0\righttriarrowfill@{#1}{#2}}
\providecommand*\xRighttriarrow[2][]{\ext@arrow 07{24}0\Righttriarrowfill@{#1}{#2}}

\providecommand*\xmapsto[2][]{\ext@arrow 01{20}0\mapstofill@{#1}{#2}}
\providecommand*\xleftrightarrow[2][]{\ext@arrow 10{22}0\leftrightarrowfill@{#1}{#2}}
\providecommand*\xLeftrightarrow[2][]{\ext@arrow 10{27}0{\Leftrightarrowfill@}{#1}{#2}}

\makeatother

\numberwithin{equation}{section}

\theoremstyle{plain}
\newtheorem{Thm}{Theorem}
\newtheorem{Prop}[Thm]{Proposition}
\newtheorem{Cor}[Thm]{Corollary}
\newtheorem{Lemma}[Thm]{Lemma}

\theoremstyle{definition}
\newtheorem{Defn}[Thm]{Definition}

\newtheorem{Ex}[Thm]{Example}
\newtheorem{Exs}[Thm]{Examples}
\newtheorem{Rk}[Thm]{Remark}

\newcommand{\atwo}{{\mathbf 2}}

\begin{document}
\leftmargini=2em 
\title{An embedding theorem for tangent categories}
\author{Richard Garner} 
\address{Department of Mathematics, Macquarie University, NSW 2109, Australia} 
\email{richard.garner@mq.edu.au}

\date{\today}

\thanks{The support of Australian Research Council Discovery
  Projects DP110102360, DP160101519 and FT160100393 is gratefully acknowledged.}

\begin{abstract}
  \emph{Tangent categories} were introduced by Rosick\'y as a
  categorical setting for differential structures in algebra and
  geometry; in recent work of Cockett, Crutwell and others, they have
  also been applied to the study of differential structure in computer
  science. In this paper, we prove that every tangent category admits
  an embedding into a representable tangent category---one whose
  tangent structure is given by exponentiating by a free-standing
  tangent vector, as in, for example, any well-adapted model of Kock
  and Lawvere's synthetic differential geometry. The key step in our
  proof uses a coherence theorem for tangent categories due to Leung
  to exhibit tangent categories as a certain kind of enriched
  category.
\end{abstract}
\maketitle

\section{Introduction}
\emph{Tangent categories}, introduced by Rosick\'y
in~\cite{Rosicky1984Abstract}, provide an category-theoretic setting
for differential structures in geometry, algebra and computer science.
A tangent structure on a category $\C$ comprises a functor
$T \colon \C \rightarrow \C$ together with associated natural
transformations---for example a natural transformation
$p \colon T \Rightarrow 1$ making each $TM$ into an object over
$M$---which capture just those properties of the ``tangent bundle''
functor on the category $\cat{Man}$ of smooth manifolds that are
necessary to develop a reasonable abstract differential calculus. The
canonical example is $\cat{Man}$ itself, but others include the
category of schemes (using the Zariski tangent spaces), the category
of convenient manifolds~\cite{Baez2011Convenient} and, in computer
science, any model of Ehrhard and Regnier's \emph{differential
  $\lambda$-calculus}~\cite{Ehrhard2003The-differential}.

A more powerful category-theoretic approach to differential structures
is the \emph{synthetic differential geometry} developed by Kock,
Lawvere, Dubuc and
others~\cite{Dubuc1979Sur-les-modeles,Kock1981Synthetic,Moerdijk1991Models}.
It is more powerful because it presupposes more: a so-called
``well-adapted'' model of synthetic differential geometry is a
Grothendieck topos $\E$ equipped with a full embedding
$\iota \colon \cat{Man} \rightarrow \E$ of the category of smooth
manifolds, obeying axioms which, among other things, assert that the
affine line $R = \iota(\mathbb R)$ has enough nilpotent elements to
detect the differential structure. In particular, in a well-adapted
model, the tangent bundle of a smooth manifold $M$ is determined by
the cartesian closed structure of $\E$ through the equation
$\iota(TM) = \iota(M)^D$; here, $D$ is the ``disembodied tangent
vector'', which in the internal logic of $\E$ comprises the elements
of $R$ which square to zero.

Any model $\E$ of synthetic differential geometry gives rise to a
tangent category, whose underlying category comprises the
\emph{microlinear} objects~\cite[Chapter~V]{Moerdijk1991Models} of
$\E$ (among which are found the embeddings $\iota(M)$ of manifolds)
and whose ``tangent bundle'' functor is $(\thg)^D$;
see~\cite[Section~5]{Cockett2013Differential}. This raises the
question of whether any tangent category can be embedded into the
microlinear objects of a well-adapted model of synthetic differential
geometry; and while this is probably too much to ask, it has been
conjectured that any tangent category should at least be embeddable in
a \emph{representable} tangent category---one whose ``tangent bundle''
functor is of the form $(\thg)^D$. The goal of this article is to
prove this conjecture.

Our approach uses ideas of enriched category
theory~\cite{Kelly1982Basic}. By exploiting Leung's coherence
result~\cite{Leung2016The-free} for tangent categories, we are able to
describe a cartesian closed category $\E$ such that tangent categories
are the same thing as $\E$-enriched categories admitting certain
\emph{powers}~\cite[Section~3.7]{Kelly1982Basic}---a kind of
enriched-categorical limit. Standard enriched category theory then
shows that, for any small $\E$-category $\C$, the $\E$-category of
presheaves $[\C^\mathrm{op}, \E]$ is complete, cocomplete and
cartesian closed as a $\E$-category. Completeness means that, in
particular, $[\C^\mathrm{op}, \E]$ bears the powers necessary for
tangent structure; but cocompleteness and cartesian closure allow
these powers to be computed as internal homs $(\thg)^D$, so that any
presheaf $\E$-category bears \emph{representable} tangent structure.
It follows that, for any (small) tangent category $\C$, the
$\E$-categorical Yoneda embedding
$\C \rightarrow [\C^\mathrm{op}, \E]$ is a full embedding of $\C$ into
a representable tangent category.

Beyond allowing an outstanding conjecture to be settled, we believe
that the enriched-categorical approach to tangent structure has
independent value, which will be explored further in future work. In
one direction, the category $\E$ over which our enrichment exists
admits an abstract version of the \emph{Campbell--Baker--Hausdorff}
construction by which a Lie algebra can be formally integrated to a
formal group law (i.e., encoding the purely algebraic part of Lie's
theorems). Via enrichment, this construction can be transported to any
suitable $\E$-enriched category, so allowing a version of Lie theory
to be associated uniformly with any category with differential
structure. Another direction we intend to explore in future research
involves modifying the category $\E$ to capture generalised forms of
differential structure. One possibility involves \emph{non-linear} or
\emph{arithmetic} differential geometry in the sense
of~\cite{Buium1997Arithmetic}, which should involve enrichments over a
suitable category of \emph{$k$-$k$-birings} in the sense
of~\cite{Tall1970Representable,Borger2005Plethystic}. Another
possibility would be to explore ``two-dimensional Lie theory'' by
replacing the cartesian closed category $\E$ with a suitable cartesian
closed bicategory of $k$-linear categories, and considering
generalised enrichments over this in the sense
of~\cite{Garner2013Enriched}.

Besides this introduction, this paper comprises the following parts.
Section~\ref{sec:background} recalls the basic notions of tangent
category and representable tangent category, along with the coherence
result of Leung on which our constructions will rest.
Section~\ref{sec:tang-categ-as-1} extends Leung's result so as to
exhibit an equivalence between the $2$-category of tangent categories
and a certain $2$-category of
\emph{actegories}~\cite{McCrudden2000Categories}---categories equipped
with an action by a monoidal category. Section~\ref{sec:tang-categ-as}
then applies two results from enriched category theory, due to Wood
and Day, to exhibit these actegories as categories enriched over a
certain base $\E$. Then, in Section~\ref{sec:an-embedding-theorem}, we
see that this base $\E$ is complete, cocomplete and cartesian closed,
and using this, deduce that the desired embedding arises simply as the
$\E$-enriched Yoneda embedding. Finally, in
Section~\ref{sec:an-expl-pres}, we unfold the abstract constructions
to give a concrete description of the embedding of any tangent
category into a representable one.

\section{Background}
\label{sec:background}

We begin by recalling the notion of tangent category and representable
tangent category. Rosick\'y's original
definition in~\cite{Rosicky1984Abstract} requires \emph{abelian group}
structure on the fibres of the tangent bundle; with motivation from
computer science, Cockett and Crutwell weaken this
in~\cite{Cockett2013Differential} to involve only \emph{commutative
  monoid} structure, and we adopt their more general formulation here,
though our results are equally valid under the narrower
definition.

\begin{Defn}
  A \emph{tangent category} is a category $\C$ equipped with:
  \begin{enumerate}[(i)]
  \item A functor $T \colon \C \rightarrow \C$ and a natural
    transformation $p \colon T \Rightarrow \id_\C$ such that each $n$-fold
    fibre product $TX \times_{p_X} \dots \times_{p_X} TX$ exists in $\C$
    and is preserved by each functor $T^m$;\vskip0.25\baselineskip
  \item Natural transformations
    \begin{equation*}
      e \colon \id_\C \Rightarrow T \qquad m \colon T \times_p T \Rightarrow T \qquad \ell \colon T \Rightarrow TT \quad \text{and} \quad c \colon TT \Rightarrow TT\rlap{ ,}
    \end{equation*}
  \end{enumerate}
  subject to the following axioms:
  \begin{enumerate}[(i)]
    \addtocounter{enumi}{2}
  \item The maps $e_X$ and $m_X$ endow each $p_X \colon TX \rightarrow
    X$ with the structure of a commutative monoid in the slice
    category $\C / X$;\vskip0.25\baselineskip
  \item The following squares commute:
    \begin{equation*}
      \cd[@-0.2em]{
        {T} \ar[r]^-{\ell} \ar[d]_{p} &
        {T^2} \ar[d]^{Tp} &
        {\id_\C} \ar[r]^-{e} \ar[d]_{e} &
        {T} \ar[d]^{Te} &
        \sh{l}{0.2em}{T \times_{p} T} \ar[r]^-{\ell \times_{e} \ell} \ar[d]_{m} &
        \sh{r}{0.2em} {T^2 \times_{Tp} T^2} \ar[d]^{Tm} \\
        {\id_\C} \ar[r]^-{e} &
        {T} &
        {T} \ar[r]^-{\ell} &
        {T^2} &
        {T} \ar[r]^-{\ell} &
        {T^2}\rlap{ ;}
      }
    \end{equation*}
  \item The following squares commute:
    \begin{equation*}
      \cd[@-0.2em]{
        {T^2} \ar[r]^-{c} \ar[d]_{Tp} &
        {T^2} \ar[d]^{pT} &
        {T} \ar[r]^-{\id} \ar[d]_{Te} &
        {T} \ar[d]^{eT} &
        \sh{l}{0.2em}{T^2 \times_{Tp} T^2} \ar[r]^-{c \times_T c} \ar[d]_{Tm} &
        \sh{r}{0.2em} {T^2 \times_{pT} T^2} \ar[d]^{mT} \\
        {T} \ar[r]^-{\id} &
        {T} &
        {T^2} \ar[r]^-{c} &
        {T^2} &
        {T^2} \ar[r]^-{c} &
        {T^2}\rlap{ ;}
      }
    \end{equation*}
  \item $c^2 = \id$, $c \ell = \ell$, and the following diagrams
    commute:
    \begin{equation*}
      \cd[@-0.2em]{
        {T} \ar[r]^-{\ell} \ar[d]_-{\ell} &
        {T^2} \ar[d]^-{T\ell} &
        T^3 \ar[r]^-{Tc} \ar[d]_-{cT} & 
        T^3 \ar[r]^-{cT} &
        T^3 \ar[d]^-{Tc} &
        {T^2} \ar[r]^-{\ell T} \ar[d]_-{c} &
        {T^3} \ar[r]^-{Tc} & {T^3} \ar[d]^-{cT} \\
        {T^2} \ar[r]^-{\ell T} &
        {T^3} &
        T^3 \ar[r]^-{Tc} &
        T^3 \ar[r]^-{cT} &
        T^3 &
        {T^2} \ar[rr]^-{T\ell} & &
        {T^3}\rlap{ ;}
      }
    \end{equation*}
  \item Writing $w$ for the composite
    $T \times_p T \xrightarrow{\ell \times_e eT} T^2 \times_{Tp} T^2
    \xrightarrow{Tm} T^2$,
    each diagram of the following form is an equaliser:
    \begin{equation}\label{eq:3}
      \cd[@C+2.4em]{
        TX \times_{pX} TX \ar[r]^-{w_X} & 
        T^2X \ar@<3pt>[r]^-{Tp_X} \ar@<-3pt>[r]_-{e_X.p_X.p_{TX}} & TX\rlap{ .}
      }
    \end{equation}
  \end{enumerate}
  In the sequel we will, as in~\cite{Cockett2013Differential}, write
  $T_n X \defeq TX \times_{p_X} \dots \times_{p_X} TX$ for the
  $n$-fold fibre product of $TX$ over $X$ in any tangent category. We
  refer to these fibre products and the
  equalisers in~\eqref{eq:3} collectively as \emph{tangent limits}.
\end{Defn}
\begin{Exs}\hfill
  \label{ex:1}
  \begin{enumerate}[(i)]
  \item The category $\cat{Man}$ of smooth manifolds is a tangent
    category under the structure for which $p_X \colon TX \rightarrow X$
    is the usual tangent bundle of $X$.\vskip0.25\baselineskip
  \item Let $\E$ be any model of synthetic differential
    geometry~\cite{Kock1981Synthetic,Moerdijk1991Models} with embedding
    $\iota \colon \cat{Man} \rightarrow \E$ of the category of smooth
    manifolds. The full subcategory of $\E$ on the \emph{microlinear}
    objects~\cite[Chapter~V]{Moerdijk1991Models} is a tangent category
    under the structure which sends a microlinear $X \in \E$ to the
    exponential $X^D$ by the object
    $D = \{x \in \iota(\mathbb{R}) : x^2 = 0\}$;
    see~\cite[Section~5]{Cockett2013Differential}.
    \vskip0.25\baselineskip
  \item The category $\cat{Sch}$ of schemes over
    $\mathrm{Spec}\ \mathbb{Z}$ is a tangent category under the
    structure which sends a scheme $X$ to its Zariski tangent space
    $TX$. One way to see this is via the full embedding of $\cat{Sch}$
    into the local ring classifier
    $\cat{Zar} = \cat{Sh}(\cat{CRng}_f^\mathrm{op}, \Z)$. The generic
    local ring $R \in \cat{Zar}$ is known to be a ring of line
    type~\cite{Kock1981Synthetic}, and so the category
    $\cat{Zar}_{ml}$ of $R$-microlinear objects in $\cat{Zar}$ is
    by~\cite[Section~5]{Cockett2013Differential} a tangent category
    under the tangent functor $(\thg)^D$, where
    $D = \mathrm{Spec}\ \mathbb{Z}[x]/x^2$. Now, every affine scheme
    is microlinear by the Zariski analogue
    of~\cite[Proposition~7.1]{Moerdijk1991Models}; whereupon it
    follows easily that a general scheme is microlinear since it is a
    patching of affine schemes along open immersions, and these are
    formally \'etale
    by~\cite[Proposition~17.1.3(i)]{EGA44}. Thus
    $\cat{Sch}$ is contained in $\cat{Zar}_{ml}$; finally,
    by~\cite[Lemma~5.12.1]{Brandenburg2014Tensor}, $\cat{Sch}$ is
    closed in $\cat{Zar}_{ml}$ under the tangent functor described
    above, and moreover this coincides with the Zariski tangent
    space.\footnote{ Another way to obtain the tangent structure on
      $\cat{Sch}$ is via the join restriction category $\cat{Aff}_p$
      of affine schemes and partial maps defined on Zariski open
      subschemes. Because the tangent functor on affine schemes
      preserves open immersions, it extends to a tangent structure on
      $\cat{Aff}_p$; whence $\cat{Sch}$ is a tangent category
      by~\cite[Corollary~6.26]{Cockett2013Differential}, since it is
      the total category of the manifold completion of $\cat{Aff}_p$.}


  \vskip0.25\baselineskip
  \item The category $\cat{CRig}$ of commutative rigs (rings without
    negatives) has a tangent structure with $TA = A[x]/x^2$ and with
    $p_A \colon TA \rightarrow A$ defined by $p_A(a + bx) = a$. It
    follows that $T_2 A \cong A[x,y]/x^2, y^2, xy$ and that
    $T^2A \cong A[x,y]/x^2, y^2$, in which terms the remaining structure
    is given by:
    \begin{align*}
      e_A(a) &= a & m_A(a+bx+cy) &= a+(b+c)x \\
      \ell_A(a+bx) &= a+bxy & c_A(a+bx+cy+dxy) &= a+cx+by+dxy\rlap{ .}
    \end{align*}\vskip0.25\baselineskip
      \item The functor $T \colon \cat{CRig} \rightarrow \cat{CRig}$ of
    (iv) preserves limits and filtered colimits, and so has a left
    adjoint $S$. It is easy to see that this endows
    $\cat{CRig}^\mathrm{op}$ with tangent structure
    (see~\cite[Proposition~5.17]{Cockett2013Differential}).\vskip0.25\baselineskip

  \item Consider the category $\W$ whose objects are formal tensor
    products of the form $W_{n_1} \otimes \dots \otimes W_{n_k}$ and
    whose morphisms are rig homomorphisms between the corresponding
    actual tensor products of commutative rigs, where here
  \begin{equation*}
    W_n \defeq \mathbb N[x_1, \dots, x_n] / 
    (x_i x_j)_{1 \leqslant i \leqslant j \leqslant n}\rlap{ .}
  \end{equation*}
  (We may also write $W$ for $W_1$). The replete image $\overline \W$
  of $\W$ in $\cat{CRig}$ is closed under the tangent structure of
  (iv), since this structure satisfies $T_nA \cong W_n \otimes A$;
  transporting this restricted tangent structure across the equivalence
  $\W \simeq \overline \W$ yields one on $\W$ with
  $T_n(A) = W_n \otimes A$.
\end{enumerate}
\end{Exs}
\begin{Defn}
  \label{def:6}
  If $\C$ is a cartesian closed category, then a tangent structure on
  $\C$ is \emph{representable} if each functor $T_n$ is of the form
  $(\thg)^{D_n}$ for some $D_n \in \C$; equivalently, if
  $T \cong (\thg)^D$ for some $D \in \C$ and all finite fibre
  coproducts $D_n = D +_0 \cdots +_0 D$ exist. Here,
  $0 \colon 1 \rightarrow D$ is the map determined by the requirement
  that $p_X \colon X^D \rightarrow X$ is given by evaluation at $0$;
  by naturality of $p$ and of exponential transpose, this map is
  equally the composite
  \begin{equation*}
    1 \xrightarrow{\overline{\id_D}} D^D \xrightarrow{p_D} D\rlap{ .}
  \end{equation*}
\end{Defn}
\begin{Ex}
  \label{ex:2}
  The tangent structure in Examples~\ref{ex:1}(ii) above is
  representable, with
  $D_n = \{\vec x \in \iota(\mathbb{R})^n : x_i x_j = 0 \text{ for all
    $1 \leqslant i \leqslant j \leqslant n$}\}$. The tangent structure
  on schemes in (iii) is similarly representable, with
  $D_n = \mathrm{Spec}(\mathbb{Z}[x_1, \dots, x_n]/(x_ix_j))$. It is
  also the case that the tangent structure in (v) is representable.
  Indeed, for each $n$ we have
  $S_n \dashv T_n \colon \cat{CRig} \rightarrow \cat{CRig}$; since
  $T_nA \cong W_n \otimes A$ and tensor product in $\cat{CRig}$ is
  also coproduct, the left adjoint $S_n$ ``co-exponentiates'' by $W_n$
  in $\cat{CRig}$, and so dually is the exponential $(\thg)^{W_n}$ in
  $\cat{CRig}^\mathrm{op}$.
\end{Ex}

\begin{Defn}
  \label{def:2}
  A \emph{tangent functor} between tangent categories $\C$ and $\D$ is
  a functor $H \colon \C \rightarrow \D$ that preserves tangent
  limits---which we reiterate means each $n$-fold pullback
  $TX \times_{p_X} \dots \times_{p_X} TX$ and each
  equaliser~\eqref{eq:3}---together with a natural isomorphism
  $\varphi \colon HT \Rightarrow TH$ rendering commutative each
  diagram:
  \begin{equation}
    \label{eq:4}
    \begin{gathered}
      \cd{
        {H} \ar[r]^-{eH} \ar[d]_{He} &
        {TH} \ar[d]^{pH} &
        {HT \times_{Hp} HT} \ar[r]^-{\varphi \times_H \varphi} \ar[d]_{Hm} &
        {TH \times_{pH} TH} \ar[d]^{mH} \\
        {HT} \ar[r]_-{Hp} \ar[ur]_-{\varphi} &
        {H} &
        {HT} \ar[r]^-{\varphi} &
        {TH}
      } \\
      \cd{
        {HT} \ar[rr]^-{\varphi} \ar[d]_{H\ell} & &
        {TH} \ar[d]^{\ell H} &
        {HTT} \ar[r]^-{\varphi T} \ar[d]_{Hc} &
        {THT} \ar[r]^-{T\varphi} & 
        {TTH} \ar[d]^{cH} \\
        {HTT} \ar[r]^-{\varphi T} &
        {THT} \ar[r]^-{T \varphi T} & {TTH} &
        {HTT} \ar[r]^-{\varphi T} &
        {THT} \ar[r]^-{T \varphi} & {TTH} \rlap{ .}
      }
    \end{gathered}
  \end{equation}
  A \emph{tangent transformation} between tangent functors
  $H, K \colon \C \rightarrow \D$ comprises a natural transformation
  $\alpha \colon H \Rightarrow K$ such that
  $\varphi. \alpha T = T\alpha . \varphi \colon HT \Rightarrow TK$.
  We write $\cat{TANG}$ for the $2$-category of tangent categories.
\end{Defn}
\begin{Rk}
  \label{rk:1}
  If we drop from the definition of tangent functor the requirements
  that $H$ preserve tangent limits and that $\varphi$ be invertible,
  we obtain the notion of \emph{lax tangent functor}
  $H \colon \C \rightarrow \D$; we will make brief use of this in
  Section~\ref{sec:an-expl-pres} below.
\end{Rk}

The goal of this paper is to show that every small tangent category
admits a full embedding into a representable tangent category. Our
result relies heavily on the following coherence result, which is
Theorem~14.1 of~\cite{Leung2016The-free}:

\begin{Thm}
  \label{thm:6}
  The tangent category $\W$ of Examples~\ref{ex:1}(vi) classifies
  tangent structures; by this we mean that, to within isomorphism,
  tangent structures on a category $\C$ correspond with strong
  monoidal functors
  $\Phi \colon (\W, \otimes, \mathbb N) \rightarrow ([\C, \C], \circ,
  \id)$ sending tangent limits in $\W$ to pointwise limits in $[\C, \C]$.
\end{Thm}
\begin{proof}[Proof (sketch)]
  The monoidal category $\W$ as defined above is \emph{strict}
  monoidal, and every object is, in a unique way, the tensor of
  objects of the form $W_{n}$. This ``flexibility'' of $\W$ means that
  any strong monoidal functor $\W \rightarrow [\C, \C]$ can be
  replaced by an isomorphic strict monoidal one, and so we may
  reduce to this case.

  Suppose, then, that we have a strict monoidal,
  tangent-limit-preserving functor
  $\Phi \colon \W \rightarrow [\C, \C]$. Let us write
  $T = \Phi(W) \colon \C \rightarrow \C$, and write
  $p \colon T \Rightarrow \id$ for
  $\Phi(p_\mathbb{N}) \colon \Phi(W) \rightarrow \Phi(\mathbb{N})$.
  Since $W_n \in \W$ is the tangent limit
  $W \times_\mathbb{N} \dots \times_\mathbb{N} W$, we see that
  $\Phi(W_n)$ must be a (pointwise) pullback
  $T \times_p \dots \times_p T$ in $[\C,\C]$, which, as previously, we
  will denote by $T_n$. The definition of $\Phi$ on a general object
  of $\W$ is now forced by strict monoidality to be:
  \begin{equation}\label{eq:6}
    \Phi(W_{n_1} \otimes \cdots \otimes W_{n_k})  = 
    T_{n_1} \circ \cdots \circ T_{n_k}\rlap{ .}
  \end{equation}

  Given this equation, the images under $\Phi$ of the maps
  $e_\mathbb N, m_\mathbb N, \ell_\mathbb N$ and $c_\mathbb N$ of the
  tangent structure on $\W$ thereby provide the remaining data for a
  tangent structure on $\C$. The corresponding axioms are all
  immediate except for the requirement that $T^m$ should preserve the
  $n$-fold pullback $T \times_p \dots \times_p T$. To see that this
  holds, note that for any $A \in \W$ the square left below, being a
  tangent limit, is sent by $\Phi$ to a pointwise pullback in
  $[\C, \C]$; but in the category of squares in $\W$, it is isomorphic
  (via the symmetry maps) to the one on the right, which is thus also
  sent to a pointwise pullback; now taking $A = W^{\otimes m}$ gives
  the result.
  \begin{equation*}
    \cd{
      {W_{n+k} \otimes A} \ar[r]^-{} \ar[d]_{} \pullbackcorner &
      {W_n\otimes A} \ar[d]^{! \otimes A} \\
      {W_k\otimes A} \ar[r]^-{! \otimes A} &
      {A}
    }
    \qquad
    \cd{
      {A \otimes W_{n+k}} \ar[r]^-{} \ar[d]_{} \pullbackcorner &
      {A \otimes W_n} \ar[d]^{A \otimes !} \\
      {A \otimes W_k} \ar[r]^-{A \otimes !} &
      {A}\rlap{ .}
    } 
  \end{equation*}

  It follows that each $\Phi$ as in the statement of the theorem gives
  rise to a tangent structure on $\C$. Conversely, given a tangent
  structure on $\C$, we define the corresponding $\Phi$ on objects
  by~\eqref{eq:6}. On morphisms, we define the images of
  $p_\mathbb{N}$, $e_\mathbb{N}$, $m_\mathbb{N}$, $\ell_\mathbb{N}$
  and $c_\mathbb{N}$ to be the $p$, $e$, $m$, $\ell$ and $c$ of our
  given tangent structure; what is rather less trivial is defining
  $\Phi$ on the other morphisms of $\W$.

  First one shows that every morphism of $\W$ can be constructed from
  $p_\mathbb{N}$, $e_\mathbb{N}$, $m_\mathbb{N}$, $\ell_\mathbb{N}$
  and $c_\mathbb{N}$ using only the monoidal structure and tangent
  limits; this is done in~\cite[Proposition~9.1]{Leung2016The-free}.
  Since $\Phi$ is supposed to be strict monoidal and
  tangent-limit-preserving, this now provides a prospective definition
  of $\Phi$ on morphisms. It remains to prove that this definition
  makes $\Phi$ well-defined, functorial, strict monoidal, and
  tangent-limit-preserving; the (hard) proof of this is the content of
  Sections~12 and~13 of~\cite{Leung2016The-free}.
\end{proof}

In what follows, we will find it convenient to make use of the
following straightforward reformulation of Leung's result.

\begin{Cor}
  \label{cor:4}
  The tangent category $\W$ is the free tangent category on an object,
  in the sense that for any tangent category $\C$, the functor
  \begin{equation}\label{eq:2}
    \cat{TANG}(\W, \C) \rightarrow \C
  \end{equation}
  given by evaluation at $\mathbb N \in \W$ is an equivalence of
  categories.
\end{Cor}
\begin{proof}
  Let $\C$ be any tangent category. By Theorem~\ref{thm:6}, there is
  an essentially-unique strict monoidal tangent-limit-preserving
  functor $\Phi \colon \W \rightarrow [\C, \C]$, and it is easy to see
  that this functor is in fact a tangent functor, when $\W$ is endowed
  with its canonical tangent structure and $[\C, \C]$ with the
  pointwise one coming from $\C$. Now for any $X \in \C$, the
  evaluation functor $\mathrm{ev}_X \colon [\C, \C] \rightarrow \C$ is
  also a tangent functor, and so the composite
  \begin{equation*}
    F = \W \xrightarrow{\Phi} [\C, \C] \xrightarrow{\mathrm{ev}_X} \C
  \end{equation*}
  is a tangent functor with $F(\mathbb{N}) = X$. This shows
  that~\eqref{eq:2} is surjective on objects; it remains to prove that
  it is fully faithful. Given tangent functors
  $H, K \colon \W \rightrightarrows \C$, it is easy to see that any tangent
  transformation $\alpha \colon H \Rightarrow K$ must render
  commutative each square
  \begin{equation*}
    \cd[@C+5em]{
      {HT_{n_1} \cdots T_{n_k}} \ar[r]^-{\varphi^H_{n_1}
        T_{n_2} \cdots T_{n_k}} \ar[d]_{\alpha T_{n_1} \cdots T_{n_k}} & \cdots \ar[r]^-{T_{n_1} \cdots T_{n_{k-1}}
        \varphi^H_{n_k}} & T_{n_1} \cdots T_{n_k}H
      \ar[d]^{T_{n_1} \cdots T_{n_k}\alpha} \\
      {KT_{n_1} \cdots T_{n_k}} \ar[r]^-{\varphi^K_{n_1}
        T_{n_2} \cdots T_{n_k}}  & \cdots \ar[r]^-{T_{n_1} \cdots T_{n_{k-1}}
        \varphi^K_{n_k}} & T_{n_1} \cdots T_{n_k}K
    }
  \end{equation*}
  where we write
  $\varphi^H_{n} = \varphi^H \times_H \cdots \times_H \varphi^H$ and
  similarly for $\varphi^K_n$. As both horizontal maps are invertible,
  we see on evaluating at $\mathbb{N}$ that the component of $\alpha$
  at a general object
  $W_{n_1} \otimes \dots \otimes W_{n_k} = T_{n_1}(T_{n_2}(\cdots
  T_{n_k}(\mathbb{N})\dots))$ of $\W$ is determined by that at
  $\mathbb{N}$; whence~\eqref{eq:2} is faithful. For fullness, we must
  check that defining components in this manner from \emph{any} map
  $\alpha_\mathbb{N} \colon H(\mathbb{N}) \rightarrow K(\mathbb{N})$
  yields a tangent transformation $\alpha$. The key point is
  naturality, which will follow from the equalities in~\eqref{eq:4}
  \emph{so long as} we can show that every map in $\W$ is the
  $\mathbb{N}$-component of some transformation derived from the
  tangent structure; which follows
  by~\cite[Proposition~9.1]{Leung2016The-free}.
\end{proof}

\section{Tangent categories as actegories}
\label{sec:tang-categ-as-1}
We will need to extend Leung's result from tangent categories to the
maps between them; for this it will be convenient to deploy the notion
of actegory~\cite{McCrudden2000Categories}. If $\M$ is a
monoidal category, then an \emph{$\M$-actegory} is a category $\C$
equipped with an ``action'' functor
$\ast \colon \M \times \C \rightarrow \C$ which is associative and
unital to within coherent isomorphism. By this, we mean that it comes
endowed with natural families of isomorphisms
$\alpha \colon (M \otimes N) \ast X \rightarrow M \ast (N \ast X)$ and
$\lambda \colon I \ast X \rightarrow X$ satisfying a pentagon and a
triangle axiom, as given in~\cite[\sec 3]{McCrudden2000Categories},
for example. $\M$-actegories are the objects of a $2$-category
$\M\text-\cat{ACT}$, wherein a $1$-cell is a functor
$F \colon \C \rightarrow \D$ equipped with natural isomorphisms
$\mu \colon F(M \ast X) \rightarrow M \ast FX$ compatible with
$\alpha$ and $\lambda$; and a $2$-cell is a transformation
$\alpha \colon F \Rightarrow G$ compatible with $\mu$.


With $\W = (\W, \otimes, \mathbb N)$ given as before, we
now define a \emph{tangent $\W$-actegory} to be a $\W$-actegory
$(\C, \ast)$ for which each functor
$(\thg) \ast X \colon \W \rightarrow \C$ preserves tangent limits;
these span a full and locally full
sub-$2$-category 
$\W\text-\cat{ACT}_\mathrm{t}$ of $\W\text-\cat{ACT}$.

\begin{Thm}
  \label{thm:2}
  The $2$-category $\cat{TANG}$ is equivalent to $\W\text-\cat{ACT}_\mathrm{t}$.
\end{Thm}
\begin{proof}
  We define a $2$-functor
  $\Gamma \colon \cat{TANG} \rightarrow \W\text-\cat{ACT}_\mathrm{t}$
  as follows. First, given a tangent category $\C$, the strong
  monoidal $\Phi \colon \W \rightarrow [\C, \C]$ of
  Theorem~\ref{thm:6} transposes to a $\W$-action
  $\ast \colon \W \times \C \rightarrow \C$ which preserves tangent
  limits in its first variable. Next, given a map of tangent
  categories $F \colon \C \rightarrow \D$, consider
  (following~\cite{Kelly1974Coherence}) the category $\K$ whose
  objects are triples
  $(A \in [\C, \C], B \in [\D, \D], \alpha \colon FA \cong BF)$ and
  whose morphisms are compatible pairs of natural transformations.
  $\K$ bears a tangent structure with ``tangent bundle'' functor
  \begin{equation*}
    (A,B,\alpha) \qquad \mapsto \qquad (TA, TB, FTA
    \xrightarrow{\varphi A} TFA
    \xrightarrow{T\alpha} TBF)\rlap{ ,}
  \end{equation*}
  with remaining data inherited from the pointwise tangent
  structures on $[\C, \C]$ and $[\D, \D]$, and with axioms following
  from those for the tangent functor $F$. By initiality of $\W$, there
  is an essentially-unique tangent functor
  $H \colon \W \rightarrow \K$ sending $\mathbb N$ to
  $(\id_\C, \id_\D, \id_F)$. Since the projections from $\K$ to
  $[\C, \C]$ and $[\D, \D]$ are clearly tangent functors, this $H$
  sends each $V \in \W$ to a triple
  \begin{equation*}
    (V \ast (\thg) \in [\C, \C],\, V \ast (\thg) \in [\D, \D],\, \mu_V \colon F(V \ast \thg) \rightarrow V \ast F(\thg) \, )
  \end{equation*}
  whose third component gives the maps necessary to make $F$ into a
  morphism of $\M$-actegories $\C \rightarrow \D$. This defines
  $\Gamma$ on morphisms; the definition on $2$-cells now follows on
  replacing $\D$ by $\D^\atwo$ in the preceding construction.

  It is immediate from Theorem~\ref{thm:6} that $\Gamma$ is
  essentially surjective on objects, and so we need only show that it
  is fully faithful on $1$- and $2$-cells. So let $\C$ and $\D$ be
  tangent categories and
  $(F, \mu) \colon \Gamma \C \rightarrow \Gamma \D$ a map of the
  corresponding $\W$-actegories. The maps $\mu_{W, \thg}$ constitute a
  natural isomorphism $F(W \ast \thg) \Rightarrow W \ast F(\thg)$
  which, since $W \ast (\thg) \cong T$ in both domain and codomain,
  determines and is determined by one
  $\varphi \colon FT \Rightarrow TF$. The axioms for a map of
  $\W$-actegories now imply commutativity of the
  diagrams~\eqref{eq:4}, and so $(F, \varphi)$ will be a tangent
  functor so long as $F$ preserves tangent limits. For the pullbacks,
  we consider the diagram
  \begin{equation*}
    \cd{
      {F(W_n \ast X)} \ar[r] \ar[d]_{\mu_{W_n}} &
      {FTX \times_{FX} \dots \times_{FX} FTX} \ar[d]^{\varphi \times_{FX} \dots \times_{FX} \varphi} \\
      {W_n \ast FX} \ar[r]_-{} &
      {TFX \times_{FX} \dots \times_{FX} TFX}\rlap{ .}
    }
  \end{equation*}
  with top edge induced by the maps
  $F(\pi_i \ast X) \colon F(W_n \ast X) \rightarrow F(W \ast X)
  \cong FTX$ and bottom induced by the maps
  $\pi_i \ast FX \colon W_n \ast FX \rightarrow W \ast FX \cong
  TFX$. The square commutes by naturality of $\mu$, and our
  assumptions means that the top, left and right sides are
  isomorphisms; whence also the bottom. The argument for preservation
  of the equalisers~\eqref{eq:3} is similar, and so $(F, \varphi)$ is
  a map of tangent categories. It is moreover easily unique
  such that $\Gamma(F, \varphi) = (F, \mu)$, so that $\Gamma$ is fully
  faithful on $1$-cells; the argument on $2$-cells is similar on
  replacing $\D$ by $\D^\atwo$.
\end{proof}

\section{Tangent categories as enriched categories}
\label{sec:tang-categ-as}
We now exploit Theorem~\ref{thm:2} in order to exhibit tangent
categories as particular kinds of \emph{enriched category} in the
sense of~\cite{Kelly1982Basic}; more precisely, we construct a base
for enrichment $\E$ such that tangent categories are the same thing as
$\E$-enriched categories admitting \emph{powers} by a certain class of
objects in $\E$; here, we recall that:
\begin{Defn}
  \label{def:5}
  If $\C$ is a category enriched over the
symmetric monoidal base $\V$, then a \emph{power}
(resp.~\emph{copower}) of $X \in \C$ by $V \in \V$ is an object
$V \pitchfork X$ (resp.~$V \cdot X$) of $\C$ together with a $\V$-natural family of
isomorphisms in $\V$ as to the left or right in:
 \begin{equation*}
   \C(Y, V \pitchfork X) \xrightarrow{\cong} \V(V, \C(Y,X)) \qquad \qquad 
   \C(V \cdot X, Y) \xrightarrow{\cong} \V(V, \C(X,Y))\rlap{ .}
 \end{equation*}
 Note that, by $\V$-naturality, such isomorphisms are determined by a
 \emph{unit} map $V \rightarrow \C(V \pitchfork X, X)$ or $V
 \rightarrow \C(X, V \cdot X)$ as appropriate.
\end{Defn}
\newcommand{\pv}[1][\M]{\P{#1}} The characterisation result in
question is our Theorem~\ref{thm:3} below; it will follow from two
basic arguments in the theory of enriched categories. The first, due
to Richard Wood, identifies actegories over a small symmetric $\M$
with $\pv$-enriched categories admitting powers by representables;
here, $\pv$ is the category $[\M, \cat{Set}]$ of presheaves on
$\M^\mathrm{op}$ under Day's \emph{convolution} monoidal structure:
\begin{Defn}
  \label{def:11}
  Let $\M$ be small symmetric monoidal. The \emph{convolution}
  monoidal structure on $[\M, \cat{Set}]$ is the symmetric monoidal
  closed structure whose unit object is $yI = \V(I, \thg)$, whose
  binary tensor product and internal hom are as displayed below, and
  whose coherence data are given as in~\cite{Day1970On-closed}:
  \begin{equation}\label{eq:15}
    \begin{aligned}
      (F \otimes G)(X)
      &= \textstyle \int^{M,N \in \M} \M(M \otimes N,
       X) \times FM \times GN \\ 
       \qquad [F,G](X) &= \textstyle\int_{M \in \M} [FM, G(X
       \otimes M)]\text{ .}
    \end{aligned}
  \end{equation}
\end{Defn}
The first step in proving Wood's result uses his characterisation of
general $\pv$-enriched categories.
\begin{Lemma}[Wood]
  \label{lem:1}
  Let $\M$ be a small symmetric monoidal category. To give a
  $\pv$-enriched category $\C$ is equally to give:
  \begin{itemize}
  \item A set $\ob \C$ of objects;
  \item For each $x, y \in \ob \C$, a functor $\C(x,y) \colon \M
    \rightarrow \cat{Set}$;
  \item For each $x \in \ob \C$, an identity element $\id_x \in
    \C(x,x)(I)$;
  \item For each $x,y,z \in \ob \C$, a family of composition functions
    \begin{equation*}
      \C(x,y)(M) \times \C(y, z)(N) \rightarrow \C(x,z)(M \otimes N)
    \end{equation*}
    natural in $M,N \in \M$,
  \end{itemize}
  subject to three axioms expressing associativity
  and unitality of composition.
\end{Lemma}
\begin{proof}
  This is~\cite[Proposition~1]{Wood1978b-V-indexed}; the key point is
  to use the Yoneda lemma to deduce that maps
  $yI \rightarrow F$ out of the unit in $\pv$ are in natural bijection
  with elements of $FI$, and that maps
  $h \colon F \otimes G \rightarrow H$ out of a binary tensor product
  are in natural bijection with natural families of maps
 $\bar h_{AB} \colon FM \times GN \rightarrow H(M \otimes N)$.
\end{proof}
The following key result is essentially contained in Chapter~1,~\sec
7~of Wood's Ph.D.~thesis~\cite{Wood1976Indicial}; the proof is simple
enough for us to include here.

\begin{Prop}[Wood]
  \label{prop:1}
  Let $\M$ be a small symmetric monoidal category. There is a
  correspondence, to within isomorphism, between $\M$-actegories and
  $\pv$-categories admitting powers by representables.
\end{Prop}
\begin{proof}
  First let $\C$ be a $\pv$-category admitting powers by
  representables. As usual, we write $\C_0$ for the underlying
  ordinary category of $\C$, whose objects are those of $\C$ and whose
  hom-sets are $\C_0(x,y) = \C(x,y)(I)$. We endow $\C_0$ with an
  $\M$-action by taking $M \ast X \defeq y_M \pitchfork X$.
  Functoriality of $\ast$ follows by the functoriality of enriched
  limits; the associativity constraints are given by
  \begin{equation*}
    y_{M \otimes N} \pitchfork X \cong (y_M \otimes y_N) \pitchfork X \cong y_V \pitchfork (y_W \pitchfork X)
  \end{equation*}
  where the first isomorphism comes from the definition of the
  convolution monoidal structure, and the second is the associativity
  of iterated powers~\cite[Equation 3.18]{Kelly1982Basic}; and the unit
  constraints are analogous. This gives an assignation
  $\C \mapsto (\C_0, y_{(\thg)} \pitchfork (\thg))$ from
  $\pv$-categories admitting powers by representables to
  $\M$-actegories.

  Conversely, if $(\C_0, \ast)$ is an $\M$-actegory, then we 
  define a $\pv$-category $\C$ with objects those of $\C_0$, with
  hom-objects $\C(X,Y) \colon \M \rightarrow \cat{Set}$ given by
  $\C(X,Y)(M) = \C_0(X, M \ast Y)$, with unit elements
  $\lambda_X \in \C(X,X)(I) = \C(X,I \otimes X)$, and with composition
  maps $\C(X,Y)(M) \times \C(Y,Z)(N) \rightarrow \C(X,Z)(M \otimes N)$
  given by sending $f \colon X \rightarrow M \ast Y$ and
  $g \colon Y \rightarrow N \ast Z$ to the composite
  \begin{equation*}
    X \xrightarrow{f} M \ast Y \xrightarrow{M \ast g} M \ast (N \ast Z) \xrightarrow{\cong} (M \otimes N) \ast Z\rlap{ .}
  \end{equation*}
  It is straightforward to check that this $\C$ has powers by
  representables given by taking $y_V \pitchfork X \defeq V \ast
  X$. Finally, it is easy to see that the preceding two constructions
  are inverse to within an isomorphism.
\end{proof}

In fact, by using results of~\cite{Gordon1997Enrichment}, this
correspondence can be enhanced to an equivalence of $2$-categories.
Let us write $\pv\text-\cat{CAT}_{\pitchfork}$ for the locally full
sub-$2$-category of $\pv\text-\cat{CAT}$ whose objects are
$\pv$-categories admitting powers by representables, and whose
$1$-cells are $\pv$-functors preserving such powers.
\begin{Prop}
  \label{prop:2}
  Let $\M$ be small symmetric monoidal. The correspondence of
  Proposition~\ref{prop:1} underlies an
  equivalence of $2$-categories
  $\M\text-\cat{ACT} \simeq \pv\text-\cat{CAT}_{\pitchfork}$.
\end{Prop}
\begin{proof}
  In~\cite[\sec 3]{Gordon1997Enrichment}, the assignation
  $\C \mapsto (\C_0, y_{(\thg)} \pitchfork (\thg))$ of the preceding
  proposition is made into the action on objects of a $2$-functor
  $\pv\text-\cat{CAT}_{\pitchfork} \rightarrow \M\text-\cat{ACT}$. The
  preceding Proposition shows that this $2$-functor is essentially
  surjective on objects, and it is $2$-fully faithful
  by~\cite[Theorem~3.4]{Gordon1997Enrichment}.
\end{proof}

In particular, with $\W = (\W, \otimes, \mathbb{N})$ given as in the preceding sections,
this proposition identifies $\W$-actegories with $\pv[\W]$-categories
admitting powers by representables. What it does not yet capture are the
limit-preservation properties required of a \emph{tangent}
$\W$-actegory; for this, we require
a second basic result of enriched category theory, concerning enrichment over a
\emph{monoidally reflective} subcategory.

\begin{Defn}
  \label{def:4}
  A \emph{symmetric monoidal reflection} is an adjunction
  \begin{equation}\label{eq:8}
    \cd{
      {(\V', \otimes', I')} \ar@<-4.5pt>[r]_-{J} \ar@{}[r]|-{\bot} &
      {(\V, \otimes, I)} \ar@<-4.5pt>[l]_-{L}
    }
  \end{equation}
  in the $2$-category $\cat{SMC}$ of symmetric monoidal categories, symmetric
  (lax) monoidal functors and monoidal transformations for which $J$
  is the inclusion of a full, replete subcategory $\V' \subseteq \V$.
  We may also say that $\V'$ is \emph{monoidally reflective} in $\V$.
\end{Defn}
Any symmetric monoidal functor $F \colon \V_1 \rightarrow \V_2$
induces a ``change of base'' \mbox{$2$-functor}
$F_\ast \colon \V_1\text-\cat{CAT} \rightarrow \V_2\text-\cat{CAT}$
which sends a $\V_1$-category $\A$ to the $\V_2$-category $F_\ast \A$
with the same objects and with $(F_\ast \A)(x,y) = F(\A(x,y))$.
Similarly, any symmetric monoidal transformation
$\alpha \colon F \Rightarrow G$ between monoidal symmetric functors
induces a $2$-natural transformation
$\alpha_\ast \colon F_\ast \Rightarrow G_\ast$ between the
corresponding change of base $2$-functors. The assignations
$F \mapsto F_\ast$ and $\alpha \mapsto \alpha_\ast$ are evidently
$2$-functorial, and so any monoidal reflection~\eqref{eq:8} gives rise
to a reflection of $2$-categories
$J_\ast \colon \V'\text-\cat{CAT} \leftrightarrows \V\text-\cat{CAT}
\colon L_\ast$. It follows that:
\begin{Lemma}
  \label{lem:2}
  For any symmetric monoidal reflection as in~\eqref{eq:8}, the $2$-functor $J_\ast
  \colon \V'\text-\cat{CAT} \rightarrow \V\text-\cat{CAT}$
  induces a $2$-equivalence between $\V'\text-\cat{CAT}$ and the full and
  locally full sub-$2$-category of $\V\text-\cat{CAT}$ on those
  $\V$-categories with hom-objects in~$\V'$.
\end{Lemma}
To obtain symmetric monoidal reflections, we use Day's
\emph{reflection theorem}:

\begin{Prop}[Day]
  \label{prop:3}
  Let $(\V, \otimes, I)$ be symmetric monoidal closed, let
  $J \colon \V' \leftrightarrows \V \colon L$ exhibit $\V'$
  as a full, replete reflective subcategory of $\V$, and suppose that we have:
  \begin{equation}\label{eq:10}
    A \in \A \ \text{ and }\  V \in \V \qquad \Longrightarrow \qquad [V,A] \in \A\rlap{ .}
  \end{equation}
  Then $\V'$ is symmetric monoidal on taking $I' = LI$ and
  $A \otimes' B = L(IA \otimes IB)$, and this structures makes $\V'$
  monoidally reflective in $\V$. Furthermore, $\V'$ is closed monoidal
  with internal hom inherited from $\V$.
\end{Prop}
\begin{proof}
  This is~\cite[Theorem~1.2]{Day1972A-reflection}, and a full proof is
  given there; we sketch an alternative approach via symmetric
  \emph{multicategories}~\cite{Lambek1969Deductive}. Let $\mathbf{V}$
  be the underlying symmetric multicategory of $\V$: so we have
  $\ob \mathbf{V} = \ob \V$ and
  $\mathbf{V}(A_1, \dots, A_n; B) = \V(A_1 \otimes \dots \otimes A_n,
  B)$. Write $I \colon \mathbf{V}' \rightarrow \mathbf{V}$ for the
  full sub-multicategory on those objects from $\V'$. Of course, we
  have natural isomorphisms $\V'(LA, B) \cong \V(A, IB)$, but by
  closedness and~\eqref{eq:10}, there are more general natural
  isomorphisms:
  \begin{equation*}
    \mathbf{V'}(LA_1, \dots, LA_n; B) \cong \mathbf{V}(A_1, \dots, A_n; IB)\rlap{ ,}
  \end{equation*}
  giving an adjunction of symmetric multicategories
  $I \colon \mathbf{V'} \leftrightarrows \mathbf{V} \colon L$. We will
  now be done as long as we can show that $\mathbf{V}'$, like
  $\mathbf{V}$, is representable. Since any left adjoint multifunctor
  preserves universal multimorphisms, we have for any $A, B
  \in \V'$ a universal multimorphism
  \[A,B \xrightarrow{\varepsilon_A^{-1}, \varepsilon_B^{-1}} LIA, LIB
    \xrightarrow{L(IA \otimes IB)} L(IA \otimes IB)\]
  exhibiting $L(IA \otimes IB)$ as the binary tensor of $A$ and $B$ in
  $\mathbf{V}'$; the same argument shows that $LI$ provides a unit object.
\end{proof}

We now use the Day reflection theorem to find a monoidally reflective
subcategory of $\pv[\W]$ which encodes the preservation of limits
required for a tangent $\W$-actegory.
\begin{Prop}
  \label{prop:4}
  The full subcategory $\E \subset \pv[\W]$ on those functors
  $F \colon \W \rightarrow \cat{Set}$ which preserve tangent limits
  (in the sense of sending them to limits in $\cat{Set}$) is 
  monoidally reflective.
\end{Prop}
\begin{proof}
  Clearly $\E$ is a full, replete subcategory of $\pv[\W]$, and its
  reflectivity is quite standard; see~\cite{Freyd1972Categories}, for
  example. To show it is monoidally reflective, it thus
  suffices to verify the closure condition~\eqref{eq:10}. So given
  $F \in \pv[\W]$ and $G \in \E$, we must show that $[F,G] \in \E$;
  writing $F$ as a colimit $\colim y_{A_i}$ of representables, we
  have
  $[F,G] \cong [\colim_{i} y_{A_i}, G] \cong \lim_i [y_{A_i}, G]$, and
  since $\E$ is closed under limits in $\pv[\W]$, it now suffices to
  show that $[y_A, G] \in \E$ whenever $G \in \E$. This follows
  because $[y_A, G](\thg) \cong G(A \otimes \thg)$ is the composite of
  $G \colon \W \rightarrow \cat{Set}$ with the map of tangent categories $A \otimes (\thg)
  \colon \W \rightarrow \W$.
\end{proof}
Since each representable in $\pv[\W]$ clearly lies in $\E$, we may
write $\E\text-\cat{CAT}_\pitchfork$ to denote the locally full
sub-$2$-category of $\E\text-\cat{CAT}$ on the $\E$-categories and
$\E$-functors which admit, respectively preserve, powers by
representables. With this notation, we can now give our promised
representation of tangent categories as enriched categories.
\begin{Thm}
  \label{thm:3}
  The $2$-category $\cat{TANG}$ is equivalent to $\E\text-\cat{CAT}_\pitchfork$.
\end{Thm}
\begin{proof}
  By Lemma~\ref{lem:2}, we can identify $\E\text-\cat{CAT}$ with a
  full sub-$2$-category of $\pv[\W]\text-\cat{CAT}$; but since the
  inclusion $\E \rightarrow \pv[\W]$ preserves internal homs, an
  $\E$-category will admit powers by representables \emph{qua}
  $\E$-category just when it does so \emph{qua} $\pv[\W]$-category,
  and so we may identify $\E\text-\cat{CAT}_\pitchfork$ with the full
  sub-$2$-category of $\pv[\W]\text-\cat{CAT}_\pitchfork$ on those
  $\C$ for which each
  $\C(X,Y) \colon \W \rightarrow \cat{Set}$ preserves tangent limits.
  Transporting across the equivalence
  $\pv[\W]\text-\cat{CAT}_\pitchfork \simeq \W\text-\cat{ACT}$ of
  Proposition~\ref{prop:2}, we may thus identify
  $\E\text-\cat{CAT}_\pitchfork$ with the full sub-$2$-category of
  $\W\text-\cat{ACT}$ on those $(\C, \ast)$ for which each 
  \begin{equation*}
    \C(Y,\, (\thg) \ast X) \colon \W \rightarrow \cat{Set}
  \end{equation*}
  preserves tangent limits. By the Yoneda lemma, this is the same as asking
  that each functor $(\thg) \ast X \colon \W \rightarrow \C$
  preserves tangent limits---which is to ask that $(\C, \ast)$ be a
  tangent $\W$-actegory. So
  $\E\text-\cat{CAT}_\pitchfork \simeq \W\text-\cat{ACT}_\mathrm{t}$,
  and now
  composing with the equivalence of Theorem~\ref{thm:2} yields
  the result.
\end{proof}
\begin{Rk}
  \label{rk:2}
  It is not hard to show that, if $\C$ and $\D$ are $\E$-categories
  admitting powers by representables, then a general $\E$-functor (not
  necesssarily preserving such powers) corresponds to a lax tangent
  functor in the sense of Remark~\ref{rk:1}. We will use this fact in
  Section~\ref{sec:an-expl-pres} below.
\end{Rk}

\section{An embedding theorem for tangent categories}
\label{sec:an-embedding-theorem}
We now use the representation of tangent categories as enriched
categories to show that any small tangent category $\C$ has a full
tangent-preserving embedding into a representable tangent category.
This embedding will simply be the Yoneda embedding
$Y \colon \C \rightarrow [\C^\mathrm{op}, \E]$ of $\C$ seen as an
$\E$-enriched category; since a presheaf category always admits
powers, and since the Yoneda embedding preserves any powers that
exist, this is certainly an embedding of tangent categories, and so
all we need to show is that the tangent structure on
$[\C^\mathrm{op}, \E]$ is in fact representable. The reason that this
is true is that the monoidal structure on $\E$ is in fact
\emph{cartesian}.
\begin{Lemma}
  \label{lem:3}
  The category $\E$ of Proposition~\ref{prop:4} is complete and
  cocomplete, and has its symmetric monoidal structure given by
  cartesian product.
\end{Lemma}
\begin{proof}
  $\E$ is a small-orthogonality class in a presheaf category, so
  locally presentable, so complete and cocomplete;
  see~\cite{Adamek1994Locally}, for example. To see that its monoidal
  structure is cartesian, note first that the monoidal structure
  $(\W, \otimes, \mathbb N)$ is \emph{cocartesian}, so that each
  $A \in \W$ bears a commutative monoid structure, naturally in $A$.
  Since the restricted Yoneda embedding
  $\W^\mathrm{op} \rightarrow \E$ is strong monoidal, each
  $y_A \in \E$ bears a cocommutative comonoid structure, naturally in
  $A$; since any colimit of commutative comonoids is again a
  commutative comonoid, and since the representables are dense in $\E$, it
  follows that each $X \in \E$ has a cocommutative comonoid
  structure, naturally in $X$: which implies~\cite{Fox1976Coalgebras}
  that the monoidal structure is in fact cartesian.
\end{proof}
\begin{Cor}
  \label{cor:2}
  For any small $\E$-category $\C$, the presheaf $\E$-category
  $[\C^\mathrm{op}, \E]$ is complete, cocomplete, and cartesian closed
  as an $\E$-category.
\end{Cor}
\begin{proof}
  Since $\E$ is complete and cocomplete as an ordinary category, the
  completeness and cocompleteness of $[\C^\mathrm{op}, \E]$ as an
  $\E$-category follows from~\cite[Proposition~3.75]{Kelly1982Basic}.
  As for cartesian closedness, we must show that each $\E$-functor
  \begin{equation}
    \label{eq:11}
   (\thg) \times F \colon [\C^\mathrm{op}, \E] \rightarrow
  [\C^\mathrm{op}, \E] 
  \end{equation}
  admits a right adjoint. Now, for each $X \in \E$,
  $(\thg) \times X \colon \E \rightarrow \E$ is the $\E$-functor
  taking \emph{copowers} by $X$ and so is cocontinuous. As limits
  and colimits in functor $\E$-categories are pointwise, each
  $\E$-functor~\eqref{eq:11} is likewise cocontinuous, and so we may
  define a right adjoint $(\thg)^F$ just as in the unenriched case by
  taking:
  \begin{equation*}
    G^F(X) = [\C^\mathrm{op}, \E](\C(\thg, X) \times F, G)\rlap{ .}\qedhere
  \end{equation*}
\end{proof}
\begin{Prop}
  \label{prop:5}
  For any small $\E$-category $\C$, the tangent category corresponding
  under Theorem~\ref{thm:3} to the presheaf $\E$-category
  $[\C^\mathrm{op}, \E]$ is representable.
\end{Prop}
\begin{proof}
  This tangent category is the underlying ordinary category of
  $[\C^\mathrm{op}, \E]$ equipped with the tangent structure
  $T_nX = y_{W_n} \pitchfork X$. Since $[\C^\mathrm{op}, \E]$ is
  cartesian closed as an $\E$-category, its underlying category is
  also cartesian closed, and so we need only show that each functor
  $T_n$ is given by an exponential. Now, since $[\C^\mathrm{op}, \E]$
  is cocomplete as an $\E$-category, it admits all copowers; thus, as for any
  object $X \in [\C^\mathrm{op}, \E]$ the exponential $\E$-functor
  $X^{(\thg)} \colon [\C^\mathrm{op}, \E]^\mathrm{op} \rightarrow
  [\C^\mathrm{op}, \E]$ 
  preserves limits, in particular powers, we have for each $E \in \E$ 
  an isomorphism
  \begin{equation*}
    X^{(E \cdot 1)} \cong E \pitchfork X^1 \cong E \pitchfork X\rlap{ ,}
  \end{equation*}
  so that \emph{any} power in $[\C^\mathrm{op}, \E]$, and in
  particular each $T_n$, can be computed as an $\E$-enriched
  exponential.
\end{proof}
Combining this with the remarks that began this section, we obtain:
\begin{Thm}
  \label{thm:4}
  For any small tangent category $\C$, the $\E$-enriched Yoneda embedding $\C
  \rightarrow [\C^\mathrm{op}, \E]$ provides a full tangent-preserving embedding
  of $\C$ into a representable tangent category. 
\end{Thm}

\section{An explicit presentation}
\label{sec:an-expl-pres}
To conclude the paper, we extract an explicit presentation of
the representable tangent category $[\C^\mathrm{op}, \E]$ into which
the preceding theorem embeds each small tangent category $\C$.
Consider first the case where $\C$ is the terminal tangent category
$1$: now $[\C^\mathrm{op}, \E]$ is simply $\E$ itself \emph{qua}
$\E$-enriched category, and powers by objects of $\E$ are simply given
by the internal hom of $\E$. So $\E$ is a representable tangent
category with tangent functor
\begin{equation}\label{eq:9}
  TX = X^{y_W} \cong X(W \otimes \thg) = X(T\thg)
\end{equation}
where the isomorphism comes from the formula~\eqref{eq:15} for the
internal hom in $[\W, \cat{Set}]$, which by Proposition~\ref{prop:3}
is equally the internal hom in $\E$. Of course, the representing object
for this tangent structure is $y_W \in \E$.

Consider now the case of a general tangent category $\C$. Objects of
$[\C^\mathrm{op}, \E]$ are $\E$-enriched functors
$\C^\mathrm{op} \rightarrow \E$, which are equally $\E$-enriched
functors $\C \rightarrow \E^\mathrm{op}$. Since \emph{qua}
$\E$-category both $\C$ and $\E^\mathrm{op}$ admit powers by
representables, we may by Remark~\ref{rk:2} identify such
$\E$-functors with lax tangent functors
$\C \rightarrow \E^\mathrm{op}$; here, the tangent structure on
$\E^\mathrm{op}$ is induced by the $\E$-enriched \emph{copowers} of
$\E$ and so given by $TX = y_W \times X$ (where the product here is taken in
$\E$).

It follows that a lax tangent functor $\C \rightarrow \E^\mathrm{op}$
comprises an ordinary functor $H \colon \C \rightarrow \E^\mathrm{op}$
together with a transformation
$\varphi \colon HT \Rightarrow y_W \times H(\thg)$ in
$[\C, \E^\mathrm{op}]$ rendering commutative the diagrams
in~\eqref{eq:4}. This is equally a functor
$H \colon \C^\mathrm{op} \rightarrow \E$ together with a natural
family of maps $y_W \times HC \rightarrow H(TC)$ in $\E$, or equally by
adjointness, a natural family of maps
\begin{equation*}
  \varphi_C \colon HC \rightarrow H(TC)^{y_W} \cong H(TC)(T\thg) 
\end{equation*}
in $\E$ satisfying suitable axioms. Now, giving
$H \colon \C^\mathrm{op} \rightarrow \E$ is in turn equivalent to
giving a functor
$H \colon \C^\mathrm{op} \times \W \rightarrow \cat{Set}$ which
preserves tangent limits in its second variable; and $\varphi$ is
now equally a family of maps
\begin{equation*}
  \varphi_{C,A} \colon H(C,A) \rightarrow H(TC,TA)
\end{equation*}
natural in $C \in \C$ and $A \in \W$ and rendering commutative those
diagrams which correspond to the axioms in~\eqref{eq:4}. All told, we
see that see that objects of the $\E$-functor category
$[\C^\mathrm{op}, \E]$ are equally well \emph{tangent modules}
$\C \tor \W$ in the sense of the following definition:
\begin{Defn}
  \label{def:7}
  A \emph{tangent module}
  $\C \tor \D$ between tangent categories comprises:
  \begin{itemize}
  \item A functor
    $X \colon \C^\mathrm{op} \times \D \rightarrow \cat{Set}$
    preserving tangent limits in its second variable;
  \item A family of maps $T \colon X(C,D) \rightarrow X(TC,TD)$ which
    are natural in $C$ and $D$, and make the following diagrams
    commute for all $x \in X(C,D)$:
    \begin{equation*}
      \cd{
        {C} \ar@{~>}[r]^-{x} \ar[d]_{e_C} &
        {D} \ar[d]^{e_D} \\
        {TC} \ar@{~>}[r]^-{Tx} \ar[d]_{p_C} &
        {TD}\ar[d]^{p_D} \\
        {C} \ar@{~>}[r]^-{x} &
        {D}
      } \qquad
      \cd[@C+1em]{
        {TC \times_{p_C} TC} \ar@{~>}[r]^-{Tx \times_{x} Tx} \ar[d]_{m_C} &
        {TD \times_{p_D} TD} \ar[d]^{m_D} \\
        {TC} \ar@{~>}[r]^-{Tx} \ar[d]_{\ell_C} &
        {TD}\ar[d]^{\ell_D} \\
        {TTC} \ar@{~>}[r]^-{TTx} &
        {TTD}
      } \qquad
      \cd{
        {TTC} \ar@{~>}[r]^-{TTx} \ar[d]_{c_C}&
        {TTD} \ar[d]^{c_D} \\
        {TTC} \ar@{~>}[r]^-{TTx} &
        {TTD}\rlap{ .}
      }
    \end{equation*}
  \end{itemize}
    Here, we use the evident notation for elements of the module
    $X$, and for the action on such elements by maps in $\C$ and $\D$.
    Note that, to construct the element top centre, we use $X$'s
    preservation of tangent pullbacks in its second variable.

    A \emph{map of tangent modules} $f \colon X \rightarrow Y$ is
    a natural transformation $f \colon X \Rightarrow Y$ commuting with
    $T_X$ and $T_Y$ in the evident sense. We write
    $\cat{TMod}(\C, \D)$ for the category of tangent modules from
    $\C$ to $\D$, and endow it with a tangent structure by defining
    $TX$ to be the tangent module with components
    $(TX)(C,D) = X(C,TD)$ and with operation
    \begin{equation*}
      T_{TX} = X(C,TD) \xrightarrow{T_X} X(TC,TTD) \xrightarrow{c_D
        \circ (\thg)} X(TC,TTD)\rlap{ .}
    \end{equation*}
    The remaining data for the tangent structure on
    $\cat{TMod}(\C, \D)$ is obtained from the corresponding data in
    $\D$ by postcomposition.
  \end{Defn}
  \begin{Rk}
    \label{rk:3}
    If $X \colon \C \tor \D$ is a tangent module, then we obtain a new
    tangent category $\mathrm{coll}(X)$, the so-called \emph{collage}~\cite{Street2004Cauchy}
    of $X$, whose objects are the disjoint union of those of $\C$ and
    $\D$, whose morphism-sets are defined by:
    \begin{align*}
      \mathrm{coll}(X)(C,C') &= \C(C,C') &
      \mathrm{coll}(X)(D,D') &= \C(D,D') \\
      \mathrm{coll}(X)(C,D) &= X(C,D)&
      \mathrm{coll}(X)(D,C) &= \emptyset 
    \end{align*}
    for all $C,C' \in \C$ and $D,D' \in \D$, whose tangent functor is
    defined from the tangent functors of $\C$ and $\D$ together with
    the family of maps $X(C,D) \rightarrow X(TC,TD)$, and whose
    remaining data for the tangent structure is obtained from that in
    $\C$ and in $\D$. This $\mathrm{coll}(X)$ is 
    a \emph{bipartite} tangent category, in the sense that it admits a
    tangent functor to the arrow category $\atwo$ endowed with the
    trivial tangent structure. In fact, it is easy to see that tangent
    modules from $\C$ to $\D$ are the same as bipartite tangent
    categories $p \colon \X \rightarrow \atwo$ such that
    $p^{-1}(0) = \C$ and $p^{-1}(1) = \D$.

    To give a universal characterisation of the collage of a tangent
    module, we would have to construct the
    \emph{equipment}~\cite{Wood1982Abstract} of tangent categories,
    tangent functors and tangent profunctors; we leave consideration
    of this to future work.
  \end{Rk}
\begin{Prop}
  \label{prop:6}
  For any tangent category $\C$, the underlying tangent category of
  the $\E$-category $[\C^\mathrm{op}, \E]$ is isomorphic to
  $\cat{TMod}(\C, \W)$.
\end{Prop}
\begin{proof}
  The bijection on objects was verified above, and that on morphisms
  is equally straightforward. All that remains is to show that the
  tangent structures on $\cat{TMod}(\C, \W)$ and on $[\C^\mathrm{op},
  \E]$ coincide; which follows easily from the description above of
  the tangent structure on $\E$, and the fact that powers in a functor
  $\E$-category are computed pointwise.
\end{proof}
In particular, this result tells us that $\cat{TMod}(\C, \W)$ is a
representable tangent category; the representing object is by
Proposition~\ref{prop:5} the copower of the terminal object of
$\cat{TMod}(\C, \W)$ by
$y_W \in \E$; this is the object $D \in \cat{TMod}(\C, \W)$
given by
\begin{equation}\label{eq:12}
  D(C,A) = \W(W,A)
\end{equation}
and with $T_{D} \colon D(C,A) \rightarrow D(TC,
TA)$ given (after some calculation) by
\begin{equation}\label{eq:13}
  \begin{aligned}
    \W(W,A) &\rightarrow \W(W,W\otimes A) \\
    f & \mapsto e_A \circ f\rlap{ .}
  \end{aligned}
\end{equation}

Finally, let us give a concrete characterisation of the action of the
$\E$-enriched Yoneda embedding $\C \rightarrow [\C^\mathrm{op}, \E]$.
This sends $X \in \C$ to the $\E$-functor
$\C(\thg, X) \colon \C^\mathrm{op} \rightarrow \E$, which corresponds
to the tangent module
$YX \colon \C^\mathrm{op} \times \W \rightarrow \cat{Set}$ with
\begin{equation*}
  YX(C, A) = \C(C, A \ast X)
\end{equation*}
and with $T_{YX} \colon YX(C,A) \rightarrow YX(TC,TA)$ sending an
element $f \colon C \rightarrow A \ast X$ of $YX(C,A)$ to the element
\begin{equation*}
  TC \xrightarrow{Tf} T(A
  \ast X) = W \ast (A \ast X) \xrightarrow{\cong} (W \otimes A) \ast X
  = TA \ast X
\end{equation*}
of $YX(TC,TA)$. Putting all the above together, we obtain the
following more concrete form of the embedding theorem:
\begin{Thm}
\label{thm:5}
For any small tangent category $\C$, there is a full tangent embedding
$Y \colon \C \rightarrow \cat{TMod}(\C, \W)$ into the representable
tangent category of tangent modules from $\C$ to $\W$.
\end{Thm}

Having arrived at this concrete form of the embedding theorem, one
might be tempted to dismantle the abstract scaffolding by which it was
obtained. However, there are several reasons why this would be not
only disingenuous but positively unhelpful. In the first instance, the
concrete description is subtle enough that without the abstract
justification it would appear entirely \emph{ad hoc}. Secondly,
without the general theory behind it, a detailed proof of
Theorem~\ref{thm:5} from first principles would be rather
involved---requiring us to show by hand that $\cat{TMod}(\C, \W)$ is a
tangent category, that it is representable, and that
$Y \colon \C \rightarrow \cat{TMod}(\C, \W)$ is a fully faithful
tangent functor.

Finally, the enriched-categorical viewpoint encourages us to look at
tangent categories in a different way. For example, it is immediate
from the enriched perspective that the functor
$T \colon \C \rightarrow \C$ associated to any tangent category is in
fact a tangent functor (since it is an $\E$-enriched power functor,
and as such preserves $\E$-enriched powers); or that the $2$-category
of tangent categories and tangent functors admits all bilimits and
bicolimits. As indicated in the introduction, we hope to exploit the
full power of this viewpoint in forthcoming work.

\end{document}